\title{\bf{The Splitting Lemma in any Characteristic}}
\author{
       \bf{Gert-Martin Greuel and Gerhard Pfister}\\
 }
\date{}

\documentclass[12pt]{article}
\usepackage{amsmath, amsthm, amssymb, mathrsfs, amscd, amsthm, wasysym, amscd,amsfonts,graphicx}
\usepackage{indentfirst,url,xypic, xcolor}
\usepackage{lipsum} 
\usepackage[all]{xy}
\usepackage{url}
\usepackage[colorlinks,plainpages]{hyperref}
\usepackage{verbatim}
\usepackage{fancyvrb}
\usepackage{listings}

\usepackage[colorlinks,plainpages]{hyperref}
\hypersetup{
	colorlinks=true,
	linkcolor=blue,
	filecolor=magenta,      
	urlcolor=cyan,
	citecolor=blue
}

\DeclareMathOperator{\rank}{rank}

\DeclareMathOperator{\ord}{ord}

\DeclareMathOperator{\rsim}{\stackrel{r}{\sim}}

\newtheorem{Definition}{ Definition}[section]
\newtheorem{Theorem}[Definition]{Theorem}
\newtheorem{Remark}[Definition]{Remark}

\newtheorem{Corollary}[Definition]{Corollary}

\newtheorem{Lemma}[Definition]{Lemma}

\newcommand{\N}{{\mathbb N}}

\newcommand{\Z}{{\mathbb Z}}
\newcommand{\R}{{\mathbb R}}
\newcommand{\C}{{\mathbb C}}
\newcommand{\Q}{{\mathbb Q}}

\newcommand{\fm}{\mathfrak{m}}

\newcommand{\bo}{\boldsymbol{0}}

\newcommand{\bu}{\boldsymbol{u}}
\newcommand{\bv}{\boldsymbol{v}}

\newcommand{\bx}{\boldsymbol{x}}

\newcommand{\beps}{\boldsymbol{\varepsilon}}
\newcommand{\eps}{\varepsilon}
\newcommand{\balpha}{\boldsymbol{\alpha}}

\def\<bx{\langle \bx \rangle}

\begin{document}
\maketitle
\section*{Abstract}
We give a simple proof of the splitting lemma in singularity theory, also known as generalized Morse lemma, for formal power series over arbitrary fields. Our proof for the uniqueness of the residual part in any characteristic is new and was previously unknown in characteristic two. Beyond the formal case, we give proofs for algebraic power series and for convergent real and complex analytic power series, which are new for non-isolated singularities.

\section{Introduction} \label{sec:1}
The splitting lemma is of great importance for the classification of singularities. It says that a power series $f$ 
splits (up to a coordinate change) as the sum of a quadratic part $f^{(2)}$ plus a series $g$ of order at least three, where $f^{(2)}$ and $g$ depend on disjoint sets of variables.  
Moreover, the series $g$, the "residual part", is unique up to a  change of coordinates. 
 The splitting lemma is well kown
for  power series or  differentiable functions $f$ of order at least two 
over the real or complex numbers. 
This result, for differentiable functions, is due to René Thom, who used it for his classification of the 7 elementary catastrophes (see his book \cite{Th75}, § 5.2D). Since then it has become an indispensable tool in singularity theory, in particular for the classification of singularities. Thom's proof, especially for the uniqueness of the residual part, is fairly long and complicated and depends on a determinacy theorem of Tougeron. 
A proof for convergent power series over the complex numbers can be found in \cite{GLS07} and \cite{dJP00}. A proof  for the existence of the splitting for formal power series over algebraically closed fields of positive characteristic is sketched in  \cite{GK90}. 

In this note we give a very simple proof of the splitting lemma (existence and uniqueness of the residual part) for formal power series over an arbitrary field.  To prove the uniqueness of the residual part we use only the implicit function theorem and some tricky substitutions, in particular in characteristic two.

The existing of a splitting beyond the formal case is more involved and requires different methods. We prove it for algebraic power series over a 
real valued field (using the nested Artin
approximation theorem) and for convergent complex and real analytic power series (using the existence of a versal unfolding).

We like to emphasize that  we do not assume that the power series has an isolated singularity, as it was supposed in the classical complex analytic case. 
\medskip

In order to treat  convergent and formal power series at the same time, let  $K$ be a field of arbitrary characteristic together with a real valuation $| \ | :K \to \R_{\ge 0}$\footnote{This means $ |a| = 0 \Leftrightarrow a=0$, $ |a||b| = |ab|$, and $|a+b| \le |a|+|b|$.}. $K$ is then called a {\em  real valued field}. If  $K$  is complete w.r.t. $| \ |$ (every Cauchy sequence with respect to $| \ |$ converges in $K$) $K$ is called a {\em complete real valued field}. 
$K$ is called {\em quasi-complete} if the completion $\bar K$  of $K$ is a separable field extension of $K$. Note that in characteristic 0 every real valued field is already quasi-complete.

Examples of a complete real valued field are $\C$  
resp. $\R$ with the usual absolute value as valuation or any field $K$ with the trivial valuation ($|0| =0, |a| =1$ if $a\ne 0$). Finite fields permit only the trivial valuation.
The absolute value on the field of rational numbers makes $\Q$ a real valued field, which is not complete.
Another example is given by the field of $p$-adic numbers.
Let \mbox{$p$} be a prime number, then the map
$$ v:\Z\setminus\{0\} \to \R_{>0}\,,\quad a \mapsto p^{-m} \text{ with }
m:=\max\{k\in \N  \mid p^k \text{ divides }a\} $$
extends to a unique real valuation of $\Q$. With this
valuation, $\Q$ is a real valued field that is not complete.
 The completion of $\Q$ with resp. to  the valuation $v$ is called the
field of {\em $p$-adic numbers}.
\medskip

Let $K$ be any real valued field, not necessarily complete.  For each \mbox{${\beps}\in
  (\R_{>0})^n$}, we
define a map 
$\|\phantom{f}\|_{\beps}:K[[x_1,\dots,x_n]]\to \R_{>0}\cup  \{\infty\}$
   by setting for \mbox{$f  =  \sum_{{\balpha}\in\N^n}c_{\balpha} {\bx}^{\balpha} $}
$$\|f\|_{\beps}:= \sum_{{\balpha} \in \N^n} |c_{\balpha}|\cdot
{\beps}^{{\balpha}}\in \R_{\ge 0}\cup \{\infty\}\,.
$$
Note that \mbox{$\|\phantom{f}\|_{\beps}$} is a norm on the set of all power series $f$ with  \mbox{$\|f\|_{\beps}<\infty$}. 
A formal power series
\mbox{$f  =  \sum_{{\balpha}\in\N^n}c_{\balpha} {\bx}^{\balpha} $}
is called {\em
convergent}
iff there exists a real vector \mbox{$\beps\in(\R_{>0})^n$}
such that \mbox{$ \| f \|_{\beps} < \infty $}.
We denote by 
$$K\langle \bx\rangle = K\langle x_1,...,x_n \rangle$$
the {\em convergent power series ring} over $K$. We denote by $f(0)$ the constant term $c_{\bo}$.

$K\langle \bx\rangle$ is a Noetherian, integral and factorial local ring  with maximal ideal $\fm =  \langle x_1,...,x_n \rangle$ (see \cite[Section I.1.2]{GLS07}; the assumption that $K$ is complete w.r.t. the real valuation is not necessary, see Remark \ref{rm.alg}). 

If $K$ is $\C$ resp. $\R$ or $\Q$ (with absolute value as valuation) then $K\langle \bx\rangle $ is the usual convergent power series ring  $\C\{\bx\}$  resp. $\R\{\bx\}$ or $\Q\{\bx\}$. If $K$ is any field with the trivial valuation then $K\langle \bx\rangle$ is the formal power series ring $K[[ \bx ]]$ (cf. \cite[Section 1.1]{GLS07}  or \cite[Kapitel I]{GR71} for further details).

We mention in passing, that if the valuation $| \ |$ on $K$ is  complete, then any convergent power series $f \in K\<bx$, $\|f\|_{\beps}< \infty$, defines a continuous (even analytic) function $f: B_\eps \to K$, with $B_\eps: =
\{(a_1,...,a_n) \in K^n \mid |a_i| < \eps_i, i =1,...,n\}$ (see \cite[Bemerkung 1.1]{GR71}). This is not the case if $K$ is not complete.  However, this aspect  is not relevant for us.\\

  Two power series $f$ and $g$ in $K\langle \bx\rangle$ are called {\em right equivalent} ($f\rsim g$) if there exists a local automorphism $\varphi$ of $K\langle \bx\rangle$   such that $\varphi(f) = g$. Recall that $\varphi$ is determined by $\varphi(x_i) = \varphi_i (x_1,...,x_n) \in K\langle \bx\rangle $ for $i = 1,...,n$, and that $\det (\frac {\partial \varphi_i}{\partial x_j})$ is a unit in  $K\langle \bx\rangle$ (i.e., $\det (\frac {\partial \varphi_i}{\partial x_j} (0)) \ne0$).
  We denote  by
$$H(f):=\Big( \frac{\partial^2 f}{\partial x_i\partial x_j}(0)\Big)_{i,j=1,\ldots,n}\in \mathrm{Mat}(n\times n, K)$$
the {\em Hessian matrix} of $f$ at $0$. If $f \in \fm^2$, then the rank of $H(f)$ is invariant under right equivalence. Moreover, if 
$f^{(2)}$ denotes the 2-jet\,\footnote{\,The $k$-jet $f^{(k)}$ is the image of $f$ in $K\<bx /\fm^{k+1}$, identified with the power series terms up to (including) degree k.}  of $f$ then $H(f) = H(f^{(2)})$.  
Note that $f^{(2)}$ is a {\em quadratic form} over $K$, that is, 
 a homogeneous polynomial of degree 2 in $K[\bx]=K[x_1,...,x_n]$,
$$q(x_1,...,x_n) = \sum_{1\le i , j\le n} a_{ij}x_ix_j, \ a_{ij} \in K,$$
 and the classification of quadratic forms is the first step in the proof of the splitting lemma. 
 
 We recall some elementary facts about quadratic forms.
 Two quadratic forms $q_1$ and $q_2$ are {\em equivalent } ($q_1\sim q_2$) if  there exists a linear automorphism (a linear change of coordinates) $\varphi$ of $K[\bx]$ such that $q_2=\varphi(q_1)$.  $\varphi$  is defined by the linear forms $\varphi (x_i)= \varphi_i (x_1,...,x_n)$,  $i=1,...,n$, and then we have 
$\varphi(q) = \sum a_{ij}\varphi_i \varphi_j$.
Equivalently, we may consider the linear isomorphism $\bx \mapsto C\bx$ of $K^n$ with $C\in GL(n,K)$ an invertible matrix and $\bx$ the column vector $(x_1,...,x_n)$. 
Then 
$\varphi_i (\bx) = \sum_{i=1,...,n}c_i x_i$, where $(c_1,...,c_n)$ is the $i$-th column vector of $C$ and we have $\varphi(q)(\bx) = q(C\bx)$.
The quadratic form
$q$ can be written in matrix form as $q(\bx)= \bx^T A \bx$ (where $^T$ means transpose) with $A=(a_{ij})$  the {\em coefficient matrix} of $q$. We have 
$q(C\bx) = \bx^T (C^T A C) \bx$, that is, the coefficient matrix of $\varphi(q)$ is $ C^T A C$. 
It follows that the rank of the coefficient matrix $A$ is independent of the chosen basis and therefore an invariant of the quadratic form $q$. 

If char$(K)\ne 2$, one may replace the  coefficient matrix $A$ by the symmetric matrix $\frac{1}{2}(A + A^T)$ with the same quadratic form. Therefore one usually assumes in characteristic $\ne 2$ that $A$ is symmetric, that is, $a_{ij}=a_{ji}$. This is however not possible if char$(K) = 2$, and  the classification of quadratic  forms in characteristic $2$  differs substantially from that in characteristic $\ne 2$.
\medskip

For the proof of the splitting lemma for convergent power series we use the following version of the implicit function theorem and the Artin approximation theorems, where the assumptions on the field  $K$ differ slightly.

\begin{Theorem}[Implicit Function Theorem for Convergent Power Series]
\label{thm.implicit}
Let $K$ be a field with a real valuation. Let $x=x_1,\ldots,x_n$ and $y=y_1,\ldots,y_m$ and $f=(f_1,\ldots,f_m)\in K\langle x,y\rangle^m$ such that 
$$f(0,0)=0  \text{ and } \det \big ( \frac{\partial f_i}{\partial y_j} (0,0)\big )_{i,j = 1,...,m}\neq 0.$$
Then there exists $y(x)\in K\langle x\rangle^m$ such that
$$f(x,y(x))=0 \text{ and } y(0)=0.$$ 
\end{Theorem}
A proof of the theorem  is given \cite{N62}[Theorem 45.3].

\begin{Remark} \label{rm.alg}{\em
In \cite[Theorem 45.3]{N62} Nagata proves the Weierstra{\ss} preparation theorem (WPT)  without the  completeness assumption. The implicit function theorem can be easily
derived from the WPT without assuming that $K$ is complete (see, for example, \cite[Proof of Theorem 1.8]{GLS07}), and hence it holds for any real valued field.
Proofs of the WPT for complete real valued fields are given at several places (e.g.  \cite[Satz 2.1.4]{KPR75} or  \cite[Theorem 1.6]{GLS07}). That the completeness assumption for the WPT (as well as for the Weierstra{\ss} division theorem and the Weierstra{\ss} finiteness theorem) is not necessary, can also be seen as follows: Let $K$ be not complete and $\bar K$ the completion of $K$. In loc. cit. the authors show first that the unit $u$ and the Weierstra{\ss} polynomial $p$ that appear in WPT exist as formal power series over $K$. Then, considering $u$ and $p$  as elements of $\bar K [[\bx]]$, they show by estimates that $\|u\|_{\beps}<\infty$ and  $\|p\|_{\beps}<\infty$. But since the coefficients of $u$ and $p$ are in $K$, it follows that $u, p \in K\<bx$.}
\end{Remark}

We will use the following approximation theorems. 

\begin{Theorem}[Artin Approximation]
\label{thm.approx}
Let $K$ be a real valued field. If the characteristic of $K$ is positive, we assume additionally that $K$ is quasi-complete. 
Let $x=x_1,\ldots,x_n$ and $y=y_1,\ldots,y_m$ and $f=(f_1,\ldots,f_k)\in K\langle x,y\rangle^k$. 
Let $\bar y(x)\in K[[x]]^m$
be a formal solution of $f=0$:
$$f(x,\bar y(x))=0.$$
Then there exists for any integer $c>0$ a convergent solution $y_c(x)\in K\langle x\rangle^m$:
$$f(x,y_c(x))=0$$
such that $$\bar y(x)\equiv y_c(x) \text{ mod } \langle x\rangle^c.$$
\end{Theorem}
The Theorem was proved in characteristic $0$ by M. Artin \cite{A68}, in characteristic $p>0$ for complete valued fields
by M. Andr\'e \cite{A75} and for quasi-complete fields by K.-P. Schemmel \cite{S82}.\footnote{Schemmel proved that the property of $K$ being quasi-complete is necessary and sufficient for the Approximation Theorem to be true.}

For any field $K$ let $K\<bx^{alg}  \subset K[[\bx]]$,
$\bx = (x_1,...,x_n)$, denote the ring of {\em algebraic power series}.\footnote{An algebraic power series is a formal power series $f(\bx)$ for which a non-zero polynomial $P(\bx, t)$ exists, such that $P(x, f(x)) = 0$ holds.}  

\begin{Lemma} \label{lm.alg}  If $K$ is a real valued field,
then $K\<bx^{alg}  \subset K\<bx$. 
\end{Lemma}

\begin{proof} This can easily be seen using Theorem \ref{thm.approx} as follows. 
Assume first that $K$ is complete.
Let $f\in K\langle x_1,\ldots,x_n\rangle=:S\subset K[[x_1,\ldots,x_n]]$ be an algebraic power series and $F\in S[y]$ a polynomial such that $F(f)=0.$  Theorem \ref{thm.approx} implies that for every $c>0$ there exists $f_c\in  S$ such that $F(f_c)=0$ and
   $f\equiv f_c \text{ mod }\frak{m}^c$. Since F has only finitely many zeros we have
   $f=f_c$ for some $c$ (this follows from the Weierstra{\ss} division theorem, \cite[Theorem I.1.8]{GLS07}).
   If $K$ is not complete we can pass to its completion $\bar K$ and obtain that $f\in \bar K\langle x_1,\ldots x_n\rangle \cap K[[x_1,\ldots,x_n]]=K\langle x_1,\ldots x_n\rangle$.
\end{proof} 
   
To prove the splitting lemma for algebraic power series we use the following nested approximation theorem proved by D. Popescu \cite{P86}. 
\begin{Theorem}[Nested Artin Approximation]\footnote{Note that a similar theorem in the analytic case is wrong, see \cite{G71}. If $s_i=n$ for all $i$, the Theorem was proved by M. Artin \cite{A69}.}
\label{thm.nestapprox}
Let $K$ be any field. 
Let $x=x_1,\ldots,x_n$ and $y=y_1,\ldots,y_m$ and $f=(f_1,\ldots,f_k)\in (K\langle x,y\rangle^{alg})^k$. Let $\bar y(x)=(\bar y_1(x),\ldots,\bar y_m(x)),$ $ \bar y_i(x) \in K[[x_1,\ldots,x_{s_i}]]$ for some $s_i\leq n$,
be a formal solution of $f=0$:
$$f(x,\bar y(x))=0.$$
Then there exists for any integer $c>0$ an algebraic solution $y_c(x)=(y_{c,1}(x),\ldots,y_{c,n}(x)),$  $y_{c,i}\in K\langle x_1,\ldots,x_{s_i}\rangle^{alg}$:
$$f(x,y_c(x))=0$$
such that $$\bar y(x)\equiv y_c(x) \text{ mod } \langle x\rangle^c.$$
\end{Theorem}
\bigskip

   We apply now Artin's approximation theorem to  formulate some results that are of interest in its own  (and that are used below) for arbitrary $K\<bx$. For a non-negative integer $k$ we call $f$ {\em right $k$-determined}, if for any $g \in K\langle \bx \rangle$ with $f-g \in \fm^{k+1}$, we have  $f\sim_r g$. Moreover, we call
$$\mu(f) := \dim_K K\langle \bx \rangle / \langle \frac{\partial f}{\partial x_1},..., \frac{\partial f}{\partial x_n}\rangle$$
the {\em Milnor number} of $f$. 
If $f \in \fm^l \setminus  \fm^{l+1}$ the integer ord$(f):=l$ is called the {\em order} of $f$.

\begin{Theorem}\label{thm.deter} Let $K$ be a real valued field and   quasi-complete if char$(K)>0$.
 Let $f \in \fm^2 \subset K\langle \bx \rangle$. 
 If $\fm^{k+2}\subset\fm^2\cdot \langle \frac{\partial f}{\partial x_1},..., \frac{\partial f}{\partial x_n}\rangle$, then $f$ is
       right $(2k-\ord(f)+2)$-determined. 
\end{Theorem}

\begin{proof} For $K\langle \bx \rangle = K[[\bx]]$ this is proved in
\cite[Proof of Theorem 3]{BGM12}.  If $K$ has a non-trivial valuation we apply Theorem \ref{thm.approx} as follows.
Let $p$ be the $(2k-\ord(f)+2)$-jet of $f$ (a polynomial). Then there exists $\bar\varphi(x)=(\bar\varphi_1(x),\ldots,\bar\varphi_n(x))\in K[[\bx]]$ such that $f(\bar\varphi(x))=p$. Using Theorem \ref{thm.approx} we find $\varphi(x)=(\varphi_1(x),\ldots,\varphi_n(x))\in K\langle \bx \rangle^n$ such that 
$$f(\varphi(x))=p \text{ and } \bar\varphi(x)\equiv \varphi(x) \text{ mod } \frak{m}^2.$$
This implies that $\varphi$ defines an automorphism of $K\langle \bx \rangle$ and $f$ is
       right $(2k-\ord(f)+2)$-determined. 
\end{proof}

\begin{Corollary}\label{cor.deter} 
 If $\mu(f)<\infty$, then  $f$ is right   $(2\mu(f)-\ord(f)+2)$-determined.
\end{Corollary}
 \bigskip
 
\section{Splitting Lemma in Characteristic  $\ne 2$}
\bigskip\noindent

In this section let $K$ be a  field of characteristic different from 2. \\

The  classification of quadratic forms  in characteristic $\ne2$ is classically known (cf. \cite[Satz 2]{Wi36}):
every quadratic form $f$ is equivalent to a "diagonal" form,
$$ f \sim a_1x_1^2+...+ a_n x_n^2, \  a_i \in K.$$
If $k$ is the rank of $H(f)$, then $ f \sim a_1x_1^2+...+ a_k x_k^2$, with $a_i \ne 0$ for $i = 1,...,k$. 

Of course, if $K$ is algebraically closed or, more generally, if $ \sqrt {a_i} \in K$, then 
$ f \sim x_1^2+...+  x_k^2$. Moreover, for $K=\R$ we have $ f \sim x_1^2+...+  x_{n_0}^2 - x_{n_0+1}^2-...-  x_{n_0+n_1}^2$, $k=n_0+n_1$, (by Jacobi-Sylvester), where $(n-k,n_0,n_1)$,  the {\em signature} of $f$, is an invariant of $f$. In fact, any classification of quadratic forms for special fields can replace the quadratic part in the splitting lemma below.

We formulate now the splitting lemma as we prove it in this note. While the uniqueness of the residual part holds in general, for the existence we have to distinguish several cases. 
  
\begin{Theorem}[Formal splitting lemma in characteristic $\ne 2$]
\label{thm.fsplitne2}
Let $K$ be a real valued field.
\begin{enumerate}
\item Let $f_0, f_1\in\fm^2\subset K\<bx$ and assume that 
$$f_0 =  a_1x_1^2+\ldots+a_kx_k^2+g_0(x_{k+1},\dots,x_n) 
\rsim f_1 = a_1x_1^2+\ldots+a_kx_k^2+g_1(x_{k+1},\dots,x_n)$$
with $a_i \in K,  a_i \ne 0$ and $g_0, g_1\in \fm^3$. Then $g_0 \rsim g_1$ in $K\langle x_{k+1},...,x_n\rangle$.
\item Let $f \in K[[ \bx]]$ and    $\rank H(f)=k$. Then 
 $$ f\ \rsim \ a_1x_1^2+\ldots+a_kx_k^2+g(x_{k+1},\dots,x_n) $$
   with $a_i \in K,  a_i \ne 0$, and
   \mbox{$g\in\fm^3$}. 
   $g$ is called the {\em residual
     part\/}\index{residual part} of $f$. By 1. it is uniquely determined up to
   right equivalence in $K[[x_{k+1},...,x_n]]$.
   
The same statement holds if $f\in K\<bx$ has an isolated singularity\,\footnote{We say that  $f\in K\<bx$ has an isolated singularity if its Milnor number is finite.}, with $g \in K\langle x_{k+1},...,x_n \rangle$ unique up to right equivalence in $K\langle x_{k+1},...,x_n \rangle$.
\end{enumerate} 
\end{Theorem}

\begin{proof}
1.  Let 
 $\varphi$  be an automorphism  of $K\<bx$ such that
$\varphi (f_0)= f_1$. Then
$\varphi$  is given by  
$$\varphi(x_i) =: \varphi_i(\bx) =: l_i(\bx)+k_i(\bx), \ i=1,...,n,$$ 
with $k_i \in \fm^2$ and $l_i$ linear forms with $\det \big(\frac{\partial l_i}{\partial x_j}\big) \ne 0$. Then $\varphi (f_0)= f_1$ means
$$
  a_1\varphi_1^2+\ldots+a_k\varphi_{k}^2+ \ g_0(\varphi_{k+1},\ldots,\varphi_n) =
 a_1x_1^2+\ldots+a_kx_{k}^2 + \ g_1(x_{k+1},\ldots,x_n).
$$
Comparing the terms of  order 2 and of order $\ge 3$,  we get
\begin{align}
\tag{*}\label{*}
\begin{split}
& a_1l_1^2+\ldots+a_kl_{k}^2   =  a_1x_1^2+\ldots+a_kx_k^2 \ \text{ and}  \\
 &\sum_{ i=1}^{k}  a_ik_{i} (2l_{i} + k_{i}) + g_0(\varphi_{k+1},\ldots,\varphi_n) = g_1(x_{k+1},\ldots,x_n).
 \end{split}
 \end{align}
We set $F_i := 2l_{i} + k_{i}$, $i = 1,...,k$, and assume that there are $\psi_1,...,\psi_{k} \in K\langle x_{k+1},...,x_n \rangle$ satisfying 
$$ F_i (\psi_1,...,\psi_{k}, x_{k+1},...,x_n)=0, \ i = 1,...,k.$$
Then we  define the endomorphism $\varphi'$ of $K\langle x_{k+1},...,x_n \rangle$
by
 $$\varphi'(x_i):= \varphi_i'(x_{k+1},\ldots,x_n) := 
 \varphi_i(\psi_1,...,\psi_{k},x_{k+1},\ldots,x_n), \ i = k+1,...,n.$$ 
This kills the terms $ a_ik_{i} (2l_{i} + k_{i})$
and by (\ref{*}) we get as required
$$g_0(\varphi'_{k+1},\ldots,\varphi'_n) =
g_1(x_{k+1},\ldots,x_n).$$

We have still to show that the endomorphism 
 $\psi$ exists and that $\varphi'$ is an automorphism of  $K\langle x_{k+1},...,x_n \rangle$.

\noindent We note that $(F_1,...,F_{k}, \varphi_{k+1},...,\varphi_{n})$ 
 is an automorphism of $K\langle x_1,...,x_n \rangle$ 
 since 
 $(\varphi_1,...,\varphi_n)$ is an automorphism and since $l_i$ is the linear part of $\varphi_i$.
 It follows that 
 $$\langle F_1, ...,F_k,  \varphi_{k+1},...,\varphi_{n} \rangle = \langle x_{1} ,...,x_{n}\rangle$$
  and if we  replace $x_i$ by $\psi_i(x_{k+1},...,x_n)$ for $i=1,...,k$, we get 
  $\langle 0, ..., 0, \varphi'_{k+1},...,\varphi'_{n} \rangle = \langle x_{k+1} ,...,x_{n}\rangle$. This shows that $\varphi'$ is an automorphism of $K\langle x_{k+1},...,x_n\rangle$.
  
  To show the existence of $\psi$, 
  we want to apply the  implicit function theorem (Theorem \ref{thm.implicit}) to $F_1,...,F_{k}$. For this we must show $\det \big ( \frac{\partial F_i}{\partial x_j} (0)\big )_{i,j = 1,...,k}=
\det \big ( \frac{\partial l_i}{\partial x_j}\big )_{i,j = 1,...,k} \ne 0$. 
The quadratic terms of (\ref{*}) read
$$\ell:= a_1 l_1^2+\ldots+a_kl_{k}^2   =  a_1x_1^2+\ldots+a_kx_k^2,$$
 and  we get $\frac{\partial \ell}{\partial x_i} = 2a_ix_{i}$  for $i \le k$. 
  Since $\frac{\partial \ell}{\partial x_i} \in \langle l_{1} ,...,l_{k}\rangle$,
  it follows that $\langle x_{1} ,...,x_{k}\rangle \subset \langle l_{1} ,...,l_{k}\rangle$ and hence, with $l'_i(x_1,...,x_{k}) 
  := l_i (x_{1} ,...,x_{k},0,...,0)$, 
  we get $\langle x_{1} ,...,x_{k}\rangle = \langle l'_{1} ,...,l'_{k}\rangle$. Hence $\det \big ( \frac{\partial l_i}{\partial x_j}\big )_{i,j = 1,...,k} = \det \big ( \frac{\partial l'_i}{\partial x_j}\big )_{i,j = 1,...,k} \ne 0$.\\

2. As  the coefficient matrix of the 2-jet of $f$  has rank $k$, the $2$-jet of $f$ can   be transformed into \mbox{$a_1x_1^2+\ldots+a_k x_k^2$} by a linear change of  coordinates, as follows from  the classification of quadratic forms mentioned above. Hence,  we can assume that 
$ f(\bx)=a_1x_1^2+\ldots+a_kx_k^2+g(x_{1},\dots,x_n)$, $g\in \fm^3$.

We prove the existence of the splitting for formal power series by constructing coordinate changes of increasing order,  providing a splitting modulo increasing powers of $\fm$, which converge formally to a formal coordinate change.
To do so write $f$ as
   $$ f(\bx)=a_1x_1^2+\ldots+a_kx_k^2+f_3(x_{k+1},\dots,x_n)+\sum_{i=1}^k   x_i\cdot g_i(x_1,\dots,x_n)\,, $$
   with $a_i\ne0$, \mbox{$g_i\in\fm^2$}, and \mbox{$f_3\in\fm^3$}. The coordinate
   change \mbox{$x_i\mapsto x_i-\frac{1}{2a_i}g_i$} for
  $i=1,\dots,k$, and \mbox{$x_i\mapsto x_i$} for
   \mbox{$i>k$}, yields
 $$ f(\bx)=a_1x_1^2+\ldots+a_kx_k^2+f_3(x_{k+1},\dots,x_n)+f_4(x_{k+1},\dots,x_n)+   \sum_{i=1}^kx_i\cdot h_i(\bx)\,, $$
   with \mbox{$h_i\in\fm^3$}, \mbox{$f_4\in\fm^4$}. Continuing with
   $h_i$ instead of $g_i$ in the same manner, the last sum will be of
   arbitrary high order, \mbox{hence $0$} in the limit. Moreover, the composition of the coordinate changes converges (in the $\fm$-adic topology) to a formal coordinate change such that the residual part
   \mbox{$g(x_{k+1},\dots,x_n)$} in the theorem is a formal power
   series.    
   
It remains to prove that if $f$ has an isolated singularity, then   
   $$ f\ \rsim \ a_1x_1^2+\ldots+a_kx_k^2+g(x_{k+1},\dots,x_n) $$
   with $a_i \in K,  a_i \ne 0$, and
   \mbox{$g\in\fm^3\cap K\langle x_{k+1},\ldots,x_n\rangle$}. This is a consequence of the following Theorem \ref{thm.asplitne2} since
   $f$ is right equivalent to a polynomial by Corollary \ref{cor.deter} (hence an algebraic power series) and the fact that algebraic power series are convergent (Lemma \ref{lm.alg}). 
\end{proof}

We prove now the splitting lemma for the ring of algebraic power series $K\<bx^{alg}$.
We call a $K$-algebra automorphism of  $K\<bx^{alg} $ an algebraic coordinate change. 

 \begin{Theorem}[Algebraic splitting lemma in characteristic $\ne 2$]
\label{thm.asplitne2}
Let  $K$ be any field and $f \in \fm^2$  be an algebraic power series with   $\rank H(f)=k$. Then
 $$ f\ \rsim \ a_1x_1^2+\ldots+a_kx_k^2+g(x_{k+1},\dots,x_n), $$
 with $a_i \in K,  a_i \ne 0$. The right equivalence is given by an algebraic coordinate change, and the residual part
   \mbox{$g\in\fm^3$} is an algebraic power series. 
   $g$  is uniquely determined up to
   right equivalence in
   $K\langle x_{k+1},...,x_n\rangle^{alg}$.
\end{Theorem}
\begin{proof}
Theorem \ref{thm.fsplitne2} implies that
$$f(\bar\varphi_1(x),\ldots,\bar\varphi_n(x))=a_1x_1^2+\ldots + a_kx_k^2 +\bar g(x_{k+1},\ldots , x_n)$$
for a suitable $\bar g\in  K[[x_{k+1},\ldots,x_n]]$ of order $\ge 3$, and an automorphism $\bar\varphi=(\bar\varphi_1(x),\ldots,\bar\varphi_n(x))$ of $K[[x]]$. From Theorem \ref{thm.nestapprox} we deduce the existence of $\varphi(x)=(\varphi_1(x),\ldots,\varphi_n(x))\in (K\<bx^{alg})^n$ and $g(x_{k+1},\ldots,x_n)\in K\langle x_{k+1},\ldots,x_n\rangle^{alg}$ such that
$$f(\varphi_1(x),\ldots,\varphi_n(x))=a_1x_1^2+\ldots + a_kx_k^2 +g(x_{k+1},\ldots , x_n)$$
and
$$\bar\varphi\equiv\varphi \text{ mod } \frak{m}^4,  \bar g \equiv g \text{ mod } \frak{m}^4.$$
This implies that $\varphi$ is an automorphism of $K\<bx^{alg}$ and $g\in \frak{m}^3$.\\
Now assume that $$a_1x_1^2+\ldots+a_kx_k^2+g_1(x_{k+1},\dots,x_n)\sim a_1x_1^2+\ldots+a_kx_k^2+g_2(x_{k+1},\dots,x_n)$$
for $g_1,g_2\in K\langle x_{k+1},\ldots,x_n\rangle^{alg}$.
Theorem \ref{thm.fsplitne2} implies that $g_1$ and $ g_2$ are right equivalent in $K[[x_{k+1},\ldots,x_n]]$. Theorem \ref{thm.nestapprox}
implies that they are also right equivalent in $K\langle x_{k+1},\ldots,x_n\rangle^{alg}$.
\end{proof}

\begin{Corollary} \label{cor.splitne2}
Let $K$ be an arbitrary field of characteristic $\ne 2$ and $K\<bx$ be either $K[[\bx]]$ or $K\<bx^{alg}$  and $f \in \fm^2$.
Assume moreover that $K$ coincides with its subfield $K^2$ of squares. Then 
  $$ f\ \rsim x_1^2+\ldots+x_k^2+g(x_{k+1},\dots,x_n), $$
with $g$  uniquely determined up to  right equivalence. 
The same  holds for $f\in K\<bx$, with $K$ a real valued field, provided $f$ has an isolated singularity.
\end{Corollary}
\begin{proof}
If $K^2 = K$, then $\sqrt{a_i} \in K$ (this is the only condition we need) and the quadratic form  $a_1x_1^2+\ldots+a_kx_k^2$ is equivalent to $x_1^2+\ldots+x_k^2$. The result follows from Theorems \ref{thm.fsplitne2} and \ref{thm.asplitne2}.
\end{proof}

 \bigskip
 We are now going to prove the splitting lemma in the analytic case. An element $f$ in $K\<bx$,
 $K \in \{\C, \R\}$ with absolute value as valuation, is called an {\em analytic power series} and an automorphism of $K\<bx$ an analytic coordinate change.
 
\begin{Theorem}[Analytic splitting lemma]
\label{thm.ansplit}
Let  $K$ be either $\C$ or $\R$ and $K\<bx$ the usual convergent power series ring. Let
$f \in \fm^2 \subset K\<bx$  with   $\rank H(f)=k$. Then
 $$ f\ \rsim \ a_1x_1^2+\ldots+a_kx_k^2+g(x_{k+1},\dots,x_n), $$
 given by an analytic coordinate change, 
  with $a_i =1$ if $K=\C$ and   $a_i = \pm 1$ if $K=\R$, and 
   \mbox{$g\in\fm^3$}. The residual part $g$ is an analytic power series that   is uniquely determined up to
   right equivalence in
   $K\langle x_{k+1},...,x_n\rangle$.
\end{Theorem}

We do not require that $f$ hs an isolated singularity. For the proof we need the notion of (versal) unfoldings, which can be formulated for any real valued field $K$.
 
 \begin{Definition} \label {def.morph}
 Let $f \in \fm \subset K\<bx =K\langle x_1,...,x_n\rangle$. 
 \begin{enumerate}
 \item An {\em (r-parameter) unfolding} of $f$ is a power series 
 $F \in K\langle \bx, \bu \rangle$, $\bu= (u_1,...,u_r)$, such that $F(\bx, \bo) =f$. \\
If $F(\bx,\bu) = f(\bx)$ we say that $F$ is the {\em constant unfolding} of $f$.
 \item Let $F \in K\langle \bx, \bu\rangle$ and $G \in K\langle \bx, \bv\rangle$, $\bv = (v_1,...,v_s)$,
 be two unfoldings of $f$. 
 A {\em morphism from $F$ to $G$} is given by a pair
 $$(\Phi, \alpha) : F \to G,$$

 \begin{enumerate}
  \item [(i)] $\Phi : K\langle \bx, \bv\rangle \to K\langle \bx, \bu\rangle$, a $K$-algebra morphism with \\
  $\Phi(x_i) =: \varphi_i \in K\langle \bx, \bu\rangle$, $\varphi_i(\bx,\bo) = x_i$, $i=1,...,n$,\\
   $\Phi(v_j) =: \phi_j \in K\langle \bu\rangle$, $\phi_j(\bo) = 0$, $\ j=1,...,s$,  and   
 \item [(ii)] $\alpha \in  K\langle \bu \rangle$, $\alpha(\bo) = 0$,
  \end{enumerate}
  such that 
  $F(\bx,\bu)  = \Phi (G) (\bx,\bu) + \alpha(\bu)$.
\medskip

 If this holds, we say that {\em $F$ is induced  from $G$} by $\Phi$.  If $\Phi$ is an isomorphism (which implies  $r=s$), we say that $F$ and $G$ are {\em isomorphic unfoldings}.
\item  An unfolding $F$ of $f$ is called {\em  versal}  (or {\em complete} in the terminology of \cite{GLS07})  if any unfolding of $f$ is induced from $F$ by a suitable morphism. A versal unfolding $F$ with minimal number of parameters $r$ is called {\em miniversal}. 
 \end{enumerate}
 \end{Definition} 

The following theorem is due to J. Mather for  $C^\infty$-functions.

\begin{Theorem} \label{thm.mather}
Let $K$ be $\R$ or $\C$ with the usual real valuations and $f\in \fm \subset K\<bx$.
\begin{enumerate}
\item The singularity $f$ has a versal unfolding if and only if $\mu(f)$ is finite.
\item Two $r$-parameter versal unfoldings of $f$ are isomorphic.
\item Every versal unfolding of $f$  is isomorphic to a constant unfolding  of a miniversal unfolding of $f$.
\item If $\{b_1,...,b_r\} \subset \fm$ represents a $K$-basis of 
$\fm/\langle \frac{\partial f}{\partial x_1},..., \frac{\partial f}{\partial x_n}\rangle$,
then
$$F(\bx,\bu) = f(\bx) + u_1b_1(\bx)+...+u_rb_r(\bx)$$
is a miniversal unfolding of $f$.
\end{enumerate}
\end{Theorem}
\begin{proof}
For a proof in the $C^\infty$-case see \cite{Br75}. It uses the existence and uniqueness of the solution of an ODE (together with the smooth dependence of the initial conditions) to prove
an infinitesimal characterization of local triviality (similar to that in \cite[Theorem 1.2.22]{GLS07}). The theorem holds thus for the real and complex analytic case as well.
\end{proof}

\begin{proof}[Proof of Theorem \ref{thm.ansplit}] As in the formal case we can assume that
  $ f(\bx)=q + g$, where
  $$q(x_1,...,x_k) = a_1x_1^2+\ldots+a_kx_k^2, \ a_i \in K \smallsetminus \{0\},$$ 
(in fact, $a_i=1$ if $K=\C$ and $a_i= \pm 1$ if $K=\R$)  and $g(x_1,...,x_n) \in \fm^3$. 
Setting $g'(x_1,...,x_k):= g(x_1,...,x_k,0,...,0)$ we get
 $g(x_{1},\dots,x_n)=g'(x_1,...,x_k) + h(x_1,...,x_n)$, with 
 $h= \sum_{i=k+1}^n x_ih_i(x_1,...,x_n)$. Then $f = f'+ h,$ $h \in \fm^3$, 
  $$f'(x_1,...,x_k):=  f(x_1,...,x_k,0,...,0) =q(x_1,...,x_k) + g'(x_1,...,x_k).$$
 $q$ has Milnor number 1 and hence is 2-determined by
Corollary \ref{cor.deter}.
Since $g' \in \fm^3$, it follows that $f'$
 is right equivalent to $q$ by an automorphism $\varphi$ of $K\langle x_1,...,x_k\rangle$. Setting $\varphi_i(x_1,...,x_k) := \varphi(x_i)$,
we have
 $$\varphi(f') = q(\varphi_1,...,\varphi_k)+g'(\varphi_1,...,\varphi_k) = q(x_1,...,x_k).$$
 Now define the automorphism  $\psi$ of $K\langle x_1,...,x_n\rangle$ by $x_i\mapsto \varphi_i(x_1,...,x_k)$ for
  $i=1,\dots,k$, and \mbox{$x_i\mapsto x_i$} for $i>k$. Then 
  $\psi(f') = \varphi(f') =q$ and
 we have
  $$\psi(f) = \psi(f') + \psi(h) = q +
  \sum_{i=k+1}^n x_ih_i(\varphi_1,...,\varphi_k,x_{k+1},...,x_n). $$
 Thus, after applying $\psi$, we may assume that $f= q + h$,
 $h= \sum_{i=k+1}^n x_ih_i(x_1,...,x_n).$
  That is $f(x_1,...,x_k,0,...,0) = q$, in other words, $f$ is an unfolding of $q$.
  
  Since $\langle \frac{\partial q}{\partial x_1},..., \frac{\partial q}{\partial x_k}\rangle =\langle x_1,...,x_k\rangle$, $q$ is a miniversal unfolding of itself by Theorem \ref{thm.mather} and    $f$ is isomorphic to the $(n-k)$-parameter constant unfolding of $q$.
 Hence this constant unfolding of $q$ can be induced from $f$ by an isomorphism $(\Phi, \alpha)$ from $q$ to $f$.  By Definition  \ref{def.morph} this means that $\Phi$  is an automorphism of $K\langle x_1,...,x_n\rangle$, $\Phi(x_i)(x_1,...,x_k,0,...,0) = x_i$ for $i=1,...,k$, and $\Phi(x_i) \in K\langle x_{k+1},...,x_n\rangle, \Phi(x_i)(0)=0$, for $i=k+1,...,n$
  such that   $q = \Phi (f) + \alpha(x_{k+1},...,x_n).$  The uniqueness of $g = -\alpha$ follows from Theorem \ref{thm.fsplitne2}.
  This proves the theorem.
 \end{proof}

 \bigskip
 
\section{Splitting Lemma in Characteristic 2}
In this section let $K$ be a field of characteristic 2.
We have the following classification of quadratic forms in characteristic 2, due to C. Arf in \cite[Satz 2]{Ar41}. 
\begin{Theorem}\label{thm.quad2}
Let $f= \sum_{1\le i\le j\le n} a_{ij}x_ix_j \in K[x_1,...,x_n]$ be a quadratic form. Then $f$ is equivalent  to the following normal form:

$$ \sum_{i \text{ odd, } i=1}^{2l-1}(a_i x_i^2 +  x_ix_{i+1} +
a_{i+1} x_{i+1}^2) + \sum_{i=2l+1}^{n}d_i x_i^2, 
$$
with $a_i, d_i \in K$. Moreover,  the form $\sum_{i=2l+1}^{n}d_i x_i^2$ is uniquely determined by $f$ up to linear coordinate change. The first sum is not unique but  $2l$, the rank of $H(f)$, 
is.
\end{Theorem}

\begin{proof} Arf proves in \cite[Satz 2]{Ar41}  that $f$ is equivalent to 
$$ \sum_{i \text{ odd, } i=1}^{2l-1}(a_i x_i^2 + b_i x_ix_{i+1} +
c_{i+1} x_{i+1}^2) + \sum_{i=2l+1}^{n}d_i x_i^2, 
$$
with $a_i, b_i, c_i, d_i \in K$, $b_i \ne 0$, together with the uniqueness statements. By setting $x_i \mapsto \frac{1}{b_i}x_i$ for $i$ odd, $i = 1,...,2l-1$, and calling $\frac{a_i}{b_i^2}$ again $a_i$, we get the required form. 
\end{proof}

\begin{Corollary}\label{cor.quad2}
Assume moreover that quadratic equations are solvable in $K$. Then $f$ is equivalent  to one of the following normal forms:
$$
\begin{array}{llll}
(a) & x_1x_2+x_3x_4+\ldots+x_{2l-1}x_{2l}+x_{2l+1}^2, & 1\le  2l+1 \le n,\\
(b)& x_1x_2+x_3x_4+\ldots+x_{2l-1}x_{2l}, & 2 \le 2l \le n.
\end{array}
$$
\end{Corollary}

\begin{proof}
 For each $i$ odd, $i=1,...,2l -1$, we make the following transformations: \\
First $x_i \mapsto x_i + u x_{i+1}$ where $u$ is a solution of $a_i u^2 + u + a_{i+1} =0$, and then apply 
\mbox {$x_{i+1} \mapsto x_{i+1} + a_i x_i$ }.\\
 This transforms $a_i x_i^2 +  x_ix_{i+1} +
a_{i+1} x_{i+1}^2 \mapsto x_ix_{i+1}$. Moreover, let (after renumbering) $d_i\ne 0$ for $i= 2l+1,...,2l+k, \ k\ge 0$, and 0 else.
Setting now
$x_i \mapsto \frac{1}{\sqrt{d_i}}x_i$, if $2l+1 \le i \le 2l+k$, we get for the last sum 
$x_{2l+1}^2+...+ x_{2l+k}^2 = (x_{2l+1}+...+ x_{2l+k})^2$ with $k\ge 0$. We get the normal form (b) if $k=0$. If  $k\ge 1$ we set $ x_{2l+1}= x_{2l+1}+...+ x_{2l+k}$ and get form (a).
\end{proof}

\begin{Remark}\label{rm.quadform}{\em 
Given a quadratic form $f= \sum_{1\le i\le j\le n} a_{ij}x_ix_j \in K[x_1,...,x_n]$, Arf introduces in \cite{Ar41} the 
vector space $V= Kv_1+...+Kv_n$ with a metric defined by the bilinear form 
$v\cdot w = f(v+w)-f(v)-f(w)=\sum_{1 \le i < j \le n} a_{ij} (x_iy_j + x_jy_i)$ 
for the two vectors $v = \sum_i x_iv_i$ and  $w=\sum_i y_iv_i$. He proves then that $V$ decomposes as an orthogonal sum $V' \oplus V^*$ with 
$V^* = \{ w \mid w\cdot v = 0 \, \, \forall \,v \in V\}$ the subspace of vectors that are orthogonal to all vectors of $V$.
$V^*$ decomposes as an orthogonal sum of 1-dimensional subspaces,
$V^* = \oplus_{i=2l+1}^n \langle v_i \rangle $ and $V'$ is the direct sum of 2-dimensional subspaces, 
$V' = \oplus_{i \text{ odd, } i=1}^{2l-1} \langle v_i,v_{i+1} \rangle $ with $v_i \cdot v_{i+1} \ne 0$.
Since $\varphi  (v\cdot w) = \varphi ( v)\varphi (w)$ it
follows that $V^*$ is invariant under any linear coordinate change $\varphi$, while $V'$ is a non-invariant complement of $V^*$ of even  dimension. This proves the normal form of $f$  in Theorem \ref{thm.quad2}  and  that the form  $\sum_{i=2l+1}^{n}d_i x_i^2$  is unique up to equivalence.
}
\end{Remark}

\begin{Lemma}\label{lm.split} 
Let $K$ be a field of characteristic 2. 
The power series
$$ f= \sum_{i \text{ odd, } i=1}^{2l-1}(a_i x_i^2 + b_i x_ix_{i+1} +
a_{i+1} x_{i+1}^2) + \sum_{i=2l+1}^{n}d_i x_i^2 + h(x_1,...,x_n),
$$
in $K[[x_1,...,x_n]]$, with $a_i, b_i, d_i \in K$, $b_i\ne 0$, and $h \in \fm^3$, is right equivalent (by a coordinate change mapping $x_i \mapsto x_i$ for  $i\ge 2l+1$) to
$$ \sum_{i \text{ odd, } i=1}^{2l-1}(a_ix_i^2 + x_ix_{i+1} + a_{i+1} x_{i+1}^2) + \sum_{i=2l+1}^{n}d_i x_i^2 +h'(x_{2l+1},...,x_n),
$$
with 
$h' \in  \langle x_{2l+1},...,x_n\rangle^3$.
\end{Lemma}
\begin{proof}
Applying Theorem \ref{thm.quad2} to the 2-jet of $f$ we may assume $b_i=1$.
Write $h =  f_3 + x_1g_1 + ...+x_{2l} g_{2l}$ with $g_i(x_1,...,x_n) \in \fm^2$ and $f_3 \in  \langle x_{2l+1},...,x_n\rangle^3$ (not depending on $x_1,...,x_{2l}$). The coordinate change 
$$
\begin{array}{lll}
x_1 &\mapsto x_1 +g_{2}, &x_{2} \mapsto x_{2} +  g_{1},\\ 
x_3 &\mapsto x_3 +g_{4}, &x_{4} \mapsto x_{4} + g_{3},\\
...\\
x_{2l-1} &\mapsto x_{2l-1} +g_{2l}, \ &x_{2l} \mapsto x_{2l} +  g_{2l-1},\\  
x_i &\mapsto x_i \text{ for } i\ge 2l+1
\end{array}
$$
transforms $f$ to
$$ \sum_{i \text{ odd, } i=1}^{2l-1}(a_ix_i^2 + x_ix_{i+1} + a_{i+1} x_{i+1}^2) + \sum_{i=2l+1}^{n}d_i x_i^2 + f_3+ f_4 + (x_1h_1+...+x_{2l}h_{2l}),
$$
with suitable $h_i \in \fm^3, f_4 \in \langle x_{2l+1},...,x_n\rangle^4$. Now make the coordinate change 
as above, but with 
$h_i$ replacing $g_i$ to get 
$$\sum_{i \text{ odd, } i=1}^{2l-1}(a_ix_i^2 + x_ix_{i+1} + a_{i+1} x_{i+1}^2)  + \sum_{i=2l+1}^n d_{i}x_{i}^2  + f_3 + f_4 + f_6 +(x_1k_1+...+x_{2l}k_{2l})$$
with suitable $k_i \in \fm^5, f_6 \in \langle x_{2l+1},...,x_n\rangle^6$. 
Continuing with $k_i$  instead of $h_i$ in the same manner, the composition of the coordinate changes makes the 
sum in brackets of arbitrary high order and 
converges (in the $\fm$-adic topology) to a formal coordinate change, with the sum in brackets being 0 in the limit. This proves the lemma.
\end{proof}

\begin{Theorem}[Splitting lemma in characteristic 2] \label{thm.split2}
Let $K$ be a field of characteristic 2.  
\begin{enumerate}
\item Let $f\in \mathfrak{m}^2\subset  K[[x_1,\ldots,x_n]]$. Then there exists an $l$, $0\leq 2l\leq n$, such that $f$ is right equivalent to 
$$ \sum_{i \text{ odd, } i=1}^{2l-1}(a_i x_i^2 + x_ix_{i+1} +
a_{i+1} x_{i+1}^2) + \sum_{i=2l+1}^{n}d_i x_i^2 + h(x_{2l+1},...,x_n),
$$
with  $a_i, d_i \in K$, $h\in \langle x_{2l+1},\ldots,x_{n} \rangle^3$. 
$2l$ is the rank of the Hessian matrix of $f$ at 0. The series
$g:=  \sum_{i=2l+1}^{n}d_i x_i^2  + h(x_{2l+1},\ldots,x_{n})$ is called the {\em residual part} of $f$, it is uniquely determined up to right equivalence in $K[[x_{2l+1},\ldots,x_{n}]]$.
\item An analogous statement holds in the following cases (char$(K)=2$):\\
 (a) Let $K$ be a  
 real valued field  and let $f\in K\<bx$ have an isolated singularity. 
 Then $g \in K\langle x_{k+1},...,x_n \rangle$ and $g$ is unique up to right equivalence in $K\langle x_{k+1},...,x_n \rangle$.
 \\
 (b) Let $K$ be any field and $f \in K\<bx^{alg}$ an algebraic power series. 
 Then $g \in K\langle x_{k+1},...,x_n \rangle^{alg}$ and $g$ is unique up to right equivalence in $K\langle x_{k+1},...,x_n \rangle^{alg}$.
\end{enumerate}
\end{Theorem}

\begin{proof}
The existence of the splitting  follows  from Theorem \ref{thm.quad2} (applied to the 2-jet of $f$) and Lemma \ref{lm.split}. 
By uniqueness of $g$ we mean that if 
$$f_0 = \sum_{i \text{ odd, } i=1}^{2l-1}(a_i x_i^2 + x_ix_{i+1} +
a_{i+1} x_{i+1}^2) + g_0 \rsim f_1 = \sum_{i \text{ odd, } i=1}^{2l-1}(a_i x_i^2 + x_ix_{i+1} + a_{i+1} x_{i+1}^2) + g_1$$
then 
$g_0(x_{2l+1},\ldots,x_n)\rsim g_1(x_{2l+1},\ldots,x_n)$, with
$$g_j = \sum_{i=2l+1}^{n}d_i x_i^2 + h_j, \ h_j\in \langle x_{2l+1},\ldots,x_{n} \rangle^3,  j=0,1.$$

\noindent To see the uniqueness,
let $\varphi$  be a coordinate change given by  
$$\varphi(x_i) = \varphi_i(\bx) = l_i(\bx)+k_i(\bx), \ i=1,...,n,$$ 
with $k_i \in \fm^2$ and $l_i$ linear forms with $\det \big(\frac{\partial l_i}{\partial x_j}\big) \ne 0$, such that
$\varphi (f_0)= f_1$. This means
\begin{align}
\tag{*}\label{*}
\begin{split}
 &\sum_{i \text{ odd, } i=1}^{2l-1}(a_i \varphi_i^2 + \varphi_i\varphi_{i+1} +
a_{i+1} \varphi_{i+1}^2) +  \sum_{i=2l+1}^{n}d_i \varphi_i^2 + \ h_0(\varphi_{2l+1},\ldots,\varphi_n) =\\ 
& \sum_{i \text{ odd, } i=1}^{2l-1}(a_i x_i^2 + x_ix_{i+1} +
a_{i+1} x_{i+1}^2) + \sum_{i=2l+1}^{n}d_i x_i^2 + \ h_1(x_{2l+1},\ldots,x_n).
\end{split}
\end{align}
By the uniqueness statement in Theorem \ref{thm.quad2} we may assume,
after a linear coordinate change among the variables  $x_{2l+1},...,x_n$, that $l_{2l+1} = x_{2l+1},...,l_{n} = x_{n}$. Hence
\begin{center}
$\sum_{i=2l+1}^{n}d_i \varphi_i^2 = \sum_{i=2l+1}^{n}(d_i x_i^2 + d_ik_i^2).$
\end{center}
\noindent 
Comparing the terms of  order $\ge 3$ and setting 
$$K_i := k_i(a_ik_i +l_{i+1}+k_{i+1}) + k_{i+1}(l_i+a_{i+1}k_{i+1})$$
 we get
 $$\sum_{i \text{ odd, } i=1}^{2l-1} K_i + \sum_{i=2l+1}^{n}d_i k_i^2 \ +  \ h_0(\varphi_{2l+1},\ldots,\varphi_n)
 =  h_1(x_{2l+1},\ldots,x_n).$$
We define for $i=1,...,2l$, 
$$
\begin{array}{llll} 
&  F_i &:= l_i +a_{i+1}k_{i+1}, &\text{ if } i  \text{ is odd}\\
&  F_i &:= l_i + k_i  + a_{i-1}k_{i-1} &\text{ if } i \text{ is even}.\\
\end{array}
$$ 
Assume now that there are $\psi_1,...,\psi_{2l} \in K\langle x_{2l+1},\ldots,x_n\rangle$ satisfying 
$$F_i (\psi_1,...,\psi_{2l}, x_{2l+1},...,x_n)=0, \ i = 1,...,2l \text{ and } \psi(0)=0.$$
Then we  replace $(x_1,...,x_{2l})$ by $(\psi_1,...,\psi_{2l})$ 
and define the endomorphism $\varphi'$ of $K\langle x_{2l+1},...,x_n\rangle$
by
 $$\varphi'(x_i):= \varphi_i'(x_{2l+1},\ldots,x_n) := 
 \varphi_i(\psi_1,...,\psi_{2l},x_{2l+1},\ldots,x_n), \ i = 2l+1,...,n.$$ 
This kills  
$\sum_{i \text{ odd, } i=1}^{2l-1} K_i 
$
and we get for the terms of order $\ge 3$
$$\sum_{i=2l+1}^{n}d_i k_i^2+h_0(\varphi'_{2l+1},\ldots,\varphi'_n) =
h_1(x_{2l+1},\ldots,x_n)$$
 and thus $g_0(\varphi'_{2l+1},\ldots,\varphi'_n) =
g_1(x_{2l+1},\ldots,x_n)$ by adding $\sum_{i=2l+1}^{n}d_i x_i^2$ on both sides.
\medskip

 To prove the theorem, we have still to show that the endomorphism 
 $\psi$ exists and that $\varphi'$ is an automorphism of  $K[[x_{2l+1},...,x_n]]$.\\
Note that the map $(l_1,\varphi_2, l_3,\varphi_4, ..., l_{2l-1},\varphi_{2l}, 
 \varphi_{2l+1},...,\varphi_{n})$ is an automorphism of \mbox{$K[[x_1,...,x_n]]$} since $(\varphi_1,...,\varphi_n)$ is an automorphism\footnote{We use the following characterization of an automorphism. Let\\ $\lambda=(\lambda_1,\ldots,\lambda_n):K[[x]]\longrightarrow K[[x]]$ be an endomorphism. The following conditions are equivalent:
 \begin{enumerate}
 \item $\lambda$ is an automorphism,
 \item $ \det \big ( \frac{\partial \lambda_i(0)}{\partial x_j}\big )_{i,j = 1,...,n} \ne 0$,
  \item  $\langle\lambda_1,\ldots,\lambda_n\rangle=\langle x_1,\ldots,x_n\rangle$.
  \end{enumerate}}
  and since $l_i$ is the linear part of $\varphi_i$.
 It follows $\langle l_1,\varphi_2, l_3,\varphi_4, ..., l_{2l-1},\varphi_{2l},  \varphi_{2l+1},...,\varphi_{n} \rangle = \langle x_{1} ,...,x_{n}\rangle$ and if we  replace $x_i$ by $\psi_i(x_{2l+1},...,x_n)$ for $i=1,...,2l$, we get 
  $\langle \varphi'_{2l+1},...,\varphi'_{n} \rangle = \langle x_{2l+1} ,...,x_{n}\rangle$ since $l_i(\psi)\in\langle x_{2l+1},\ldots,x_n\rangle^2$. This shows that $\varphi'$ is an automorphism of $K[[x_{2l+1},...,x_n]]$.
  
  To show the existence of $\psi$  we  apply the  implicit function theorem (Theorem \ref{thm.implicit})  to $F_1,...,F_{2l}$
  considered as elements of $K[[x_{2l+1},\ldots,x_n]][[x_1,\ldots,x_{2l}]]$. To do this, we must show $\det \big ( \frac{\partial F_i}{\partial x_j} (0,0)\big )_{i,j = 1,...,2l}
  \ne 0$. Since $k_i\in \fm^2$ we have 
  $$\frac{\partial F_i}{\partial x_j}(0,0)=\frac{\partial l_i}{\partial x_j}(0,0)=\frac{\partial \varphi_i}{\partial x_j}(0,0).$$
Comparing the quadratic terms of (\ref{*}), we get
 $$\ell:=\sum_{i \text{ odd, } i=1}^{2l-1}(a_i l_i^2 + l_ili_{i+1} +
a_{i+1} l_{i+1}^2 ) =
 \sum_{i \text{ odd, } i=1}^{2l-1}(a_i x_i^2 + x_ix_{i+1} +
a_{i+1} x_{i+1}^2),$$
 and $\frac{\partial \ell}{\partial x_i} = x_{i+1}$ if $i$ is odd and 
  $\frac{\partial \ell}{\partial x_i} = x_{i-1}$ if $i$ is even for $i \le 2l$. 
  Since $\frac{\partial \ell}{\partial x_i} \in \langle l_{1} ,...,l_{2l}\rangle$,
  it follows that $\langle x_{1} ,...,x_{2l}\rangle \subset \langle l_{1} ,...,l_{2l}\rangle$ and hence, with $l'_i(x_1,...,x_{2l}) 
  := l_i (x_{1} ,...,x_{2l},0,...,0)$, 
  we get $\langle x_{1} ,...,x_{2l}\rangle = \langle l'_{1} ,...,l'_{2l}\rangle$. Hence $$\det \big ( \frac{\partial F_i(0,0)}{\partial x_j}\big )_{i,j = 1,...,2l}=\det \big ( \frac{\partial l_i(0,0)}{\partial x_j}\big )_{i,j = 1,...,2l} = \det \big ( \frac{\partial l'_i(0,0)}{\partial x_j}\big )_{i,j = 1,...,2l} \ne 0.$$
  Now we apply the  implicit function theorem (Theorem \ref{thm.implicit}) and obtain 
  $\psi_1,...,\psi_{2l} \in K\langle x_{2l+1},\ldots,x_n\rangle$ satisfying 
$$F_i (\psi_1,...,\psi_{2l}, x_{2l+1},...,x_n)=0, \ i = 1,...,2l \text{ and } \psi(0)=0.$$
  To prove 2.(a) note that if $f$ has an isolated singularity then $f$ is right equivalent (by a convergent coordinate change) to a polynomial, which is an algebraic powers series. This reduces  the existence of a splitting (by Lemma \ref{lm.alg}) to the case to 2.(b). The uniqueness of the residual part follows from Remark  \ref{rm.unique}.
  
  To prove 2.(b) we use 1. and get
  $$f(\bar\varphi)= \sum_{i \text{ odd, } i=1}^{2l-1}(a_i x_i^2 + x_ix_{i+1} +
a_{i+1} x_{i+1}^2) + \sum_{i=2l+1}^{n}d_i x_i^2 + \bar h(x_{2l+1},...,x_n),
$$ with $\bar\varphi=(\bar\varphi_1,\ldots,\bar\varphi_n)$ an automorphism of $K[[x_1,\ldots,x_n]]$ and $\bar h\in 
\frak{m}^3\cap K[[x_{2l+1},\ldots,x_n]]$. 
Theorem \ref{thm.nestapprox} implies the existence of 
$\varphi\in (K\langle x_1,\ldots,x_n\rangle^{alg})^n$ and $h\in K\langle x_{2l+1},\ldots,x_n\rangle^{alg}$ such that
 $$f(\varphi)= \sum_{i \text{ odd, } i=1}^{2l-1}(a_i x_i^2 + x_ix_{i+1} + a_{i+1} x_{i+1}^2) + \sum_{i=2l+1}^{n}d_i x_i^2 +  h(x_{2l+1},...,x_n) $$ 
and 
$$\bar\varphi\equiv \varphi \text{ mod } \frak{m}^4 \text{ and } \bar h\equiv h\text{ mod }\frak{m}^4.$$
This implies that $h\in\frak{m}^3$ and $\varphi$ defines an automorphism of $K\langle x_1,\ldots,x_n\rangle^{alg}$. 
Now assume given $f_j= \sum_{i \text{ odd, } i=1}^{2l-1}(a_i x_i^2 + x_ix_{i+1} + a_{i+1} x_{i+1}^2) + \sum_{i=2l+1}^{n}d_i x_i^2 +  h_j(x_{2l+1},...,x_n)$, $j=1,2$ 
such that $f_1\sim f_2$
in $K\langle x_1,\ldots,x_n\rangle^{alg}$. 
The first part of the Theorem implies that $h_1$ and $h_2$ are right equivalent in
$K[[x_{2l+1},\ldots,x_n]]$. From Theorem \ref{thm.nestapprox} we deduce that they are also right equivalent in
$K\langle x_{2l+1},\ldots,x_n\rangle^{alg}$.
\end{proof}

\begin{Remark} \label{rm.unique} {\em As in Theorem \ref{thm.fsplitne2}, the uniqueness of the residual part $g$ in Theorem \ref{thm.split2} (1) does not only hold for  $K[[\bx]]$ but for general  $K\<bx$, $K$ a real valued field of characteristic 2:\\
Let $f_0, f_1\in\fm^2\subset K\<bx$ and assume that 
$$f_0 = q(x_1,...,x_{2l})+g_0(x_{2l+1},\dots,x_n) 
\rsim f_1 = q(x_1,...,x_{2l})+g_1(x_{2l+1},\dots,x_n)$$
with 
$$ q=\sum_{i \text{ odd, } i=1}^{2l-1}(a_i x_i^2 + x_ix_{i+1} +
a_{i+1} x_{i+1}^2), \ a_i \in K,$$ and
$$g_j = \sum_{i=2l+1}^{n}d_i x_i^2 + h_j, \ h_j\in \langle x_{2l+1},\ldots,x_{n} \rangle^3, \ d_i \ne0, \ j=0,1.$$
Then $g_0 \rsim g_1$ in $K\langle x_{2l+1},...,x_n\rangle$.
   
 This follows because we use in the proof of Theorem \ref{thm.split2} (1) only the implicit function theorem, which holds for $K\<bx$.}
 \end {Remark}

Together with Corollary \ref{cor.quad2} we get from Theorem \ref{thm.split2}:
\begin{Corollary} \label{cor1.5.5}\index{splitting lemma}
Assume moreover that quadratic equations are solvable in $K$. Then there exists an $l$, $0\leq 2l\leq n$, such that $f \in \fm^2 \subset K[[\bx]]$ is right equivalent to one of the normal forms, with $g$ unique up to right equivalence:
$$
\begin{array}{llll}
(a) & x_1x_2+x_3x_4+\ldots+x_{2l-1}x_{2l}+x_{2l+1}^2 &+g(x_{2l+1},\ldots,x_{n}), & 1\le  2l+1 \le n,\\
(b)& x_1x_2+x_3x_4+\ldots+x_{2l-1}x_{2l} &+ g(x_{2l+1},\ldots,x_{n}),  & 2 \le 2l \le n.
\end{array}
$$
Moreover, statement 2. of Theorem \ref{thm.split2} holds analogously.
\end{Corollary}

{\bf Acknowledgement:} We would like to thank Thomas Preu for his questions and the anonymous referee for carefully reading the manuscript, both of which contributed to improving the presentation.

\bigskip
\noindent Gerhard Pfister\\
Department of Mathematics,
Rheinland-Pfälzische Technische Universität\\ Kaiserslautern-Landau (RPTU), Germany\\
Email address: pfister@mathematik.uni-kl.de

\medskip
\noindent Gert-Martin Greuel\\
Department of Mathematics,
Rheinland-Pfälzische Technische Universität\\ Kaiserslautern-Landau (RPTU), Germany\\
Email address: greuel@mathematik.uni-kl.de
\end{document}